\documentclass[a4paper,twoside,11pt,reqno]{amsart}
\usepackage{amssymb}
\usepackage{amsmath}
\usepackage{amsfonts}

\newtheorem{theorem}{Theorem}

\newtheorem{lemma}[theorem]{Lemma}

\newtheorem{proposition}[theorem]{Proposition}

\def\func#1{\mathop{\rm #1}}

\title{  Green's formula and variational inequations}
\author{Marc Atteia} 
\email{marcatteia@orange.fr } 
\date{Version of \today}
\begin{document}
\maketitle 
\begin{abstract}
In this paper, I give an extension to the banachic  - non linear- case of the work of B. Hanouzet and J.L. Joly
which is quoted below in the bibliography and which was elaborated in the hilbertian- linear- case.
\end{abstract}
\section{ Introduction}

 Let  $\mathcal{L}$, $\mathcal{B}$ , $\mathcal{B}_{0}$, $\mathcal{B}_{1}$, $\mathcal{T}$\, be five reflexive Banach spaces \ such that :
\begin{itemize}
\item[$\left( \ast \right) $] \ $\mathcal{B}_{0}$ (resp. $\mathcal{B}_{1}$)\ is a dense vectorial subspace of \ $\mathcal{L}$ and  $j_{0}$ (resp. \ $j_{1}$) \ is continuous injection from \  
$\mathcal{B}_{0}$ (resp. $\mathcal{B}_{1}$) \ into \ $\mathcal{L}$ .
\\
\item[$\left( \ast \ast \right) $]\ $\mathcal{B}_{0}$ \ is a
vectorial subspace of \ $\mathcal{B}_{1}$, embedded with the norm of \ $\mathcal{B}_{1}$, which is closed in \ $\mathcal{B}_{1}.$
\\
 We shall denote by \ $j$ \ the (continuous) injection of \ $\mathcal{B}_{0}$ \ into $\mathcal{B}_{1}$ .
\\
\item[$\left( \ast \ast \ast \right)$] \ ${\Gamma }$ \ is
the (closed) image of $\mathcal{B}_{1}$ \ in \ $\mathcal{T}$ \ \ by a linear
mapping \ $\gamma _{0}$ \ such that :
$$\func{Im}\,j=\ker\, \gamma _{0}.$$
\end{itemize}
 Let us denote by \ $\mathcal{L}^{\prime }$, $\mathcal{B}^{\prime }$ , 
$\mathcal{B}_{0}^{\prime }$, $\mathcal{B}_{1}^{\prime }$, $\mathcal{T}$\ $%
^{\prime }$ \ the topological duals of \ $\mathcal{L}$, $\mathcal{B}$ , $%
\mathcal{B}_{0}$, $\mathcal{B}_{1}$, $\mathcal{T}$\ 
respectively.
\\
 Then, \ $^{t}j_{0}$ (resp. $^{t}$\ $j_{1}$) \ \ is a (linear
and continuous) injective mapping from
 $\mathcal{L}^{\prime }$ $\ $into $\ \mathcal{B}_{0}^{\prime }$ \
(resp. $\mathcal{B}_{1}^{\prime })$ \ with a dense image for the strong
topology.
\\
 As \ : \ $j_{0}$ $=$\ $j_{1}j$ \ \ , we have : \ $^{t}j_{0}=$ $%
^{t}j$ $^{t}j_{1}$ . \ We set \ : \ \ $\rho =$ $^{t}j$ .
\\
 As \ $Im\left( j\right) $ is closed , \ $Im\left( ^{t}j\right) $ \ is
also closed.
\\
$^{t}\gamma _{0}$ \ \ is a (linear and continuous) injective
mapping from \ $\mathcal{T}^{\prime }$ $\ $into $\ \mathcal{B}_{1}^{\prime }$
  (embedded with its strong topology).
\\ 
Moreover, $Im\left( ^{t}\gamma _{0}\right) $ \ is closed (as \ 
$Im\left( \gamma _{0}\right) $ is closed) \ and we have :
 $$\func{Im}\, \left( ^{t}\gamma _{0}\right) =\ker \,\left( ^{t}j\right).$$
 So, we have the following scheme :
$$
 \begin{array}{ccccccc}
&& \mathcal{T}^{\prime }&\longrightarrow &\mathcal{T}&&
\\ 
&&\downarrow&&\uparrow &&
\\ 
\mathcal{L}^{\prime} & \longrightarrow &  
\mathcal{B}_{1}^{\prime} & \longrightarrow & 
\mathcal{B}_{1} & \longrightarrow & \mathcal{L}
\\  
&\searrow & \downarrow & &\uparrow & \nearrow& 
\\ 
&&\mathcal{B}_{0}^{\prime }&\longrightarrow & \mathcal{B}_{0}&&
\end{array}  
$$
 Below, we denote by \ $\left\Vert \cdot \right\Vert _{%
\mathcal{L}}$ \ (resp. \ $\left\Vert \cdot \right\Vert _{1}$ , $\left\Vert
\cdot \right\Vert _{0}$ , $\left\Vert \cdot \right\Vert _{\mathcal{T}}$ ) \
the norm of \ $\mathcal{L}$ \ 
 (resp. $\mathcal{B}_{1}$, $\mathcal{B}_{0}$, $\mathcal{T}$\ ) and by
\ \ $\left\langle \cdot ,\cdot \right\rangle _{\mathcal{L}^{\prime },\ 
\mathcal{L}}$ \ (resp. \ $\left\langle \cdot ,\cdot \right\rangle _{1}$ , \ $%
\left\langle \cdot ,\cdot \right\rangle _{0}$ , $\left\langle \cdot ,\cdot
\right\rangle _{\mathcal{T}^{\prime },\ \mathcal{T}}$) \ the duality bracket 
 between \ $\mathcal{L}$ \ and \ \ $\mathcal{L}^{\prime }$ (resp. $%
\mathcal{B}_{1}$ \ and \ \ $\mathcal{B}_{1}^{\prime }$ , \ $\mathcal{B}_{0}$
\ and \ \ $\mathcal{B}_{0}^{\prime }$ \ , $\ \mathcal{T}$ \ and \ \ $%
\mathcal{T}^{\prime }$\ \ ).\ 
\\  So, we have the following scheme :
$$
\begin{array}{ccccc}
 \mathcal{L}^{\prime } & \longrightarrow &  ^{t}j_{1}\mathcal{L}^{\prime} & \longrightarrow & \mathcal{B}_{1}^{\prime } 
\\ 
&\searrow & \uparrow && \downarrow 
\\ 
 && ^{t}j_{1}\mathcal{L}^{\prime } &   \longrightarrow & \mathcal{B}_{0}^{\prime }
 \end{array}
 $$
  The restriction of \ $\rho $ \ to \ $\ ^{t}j_{1}\mathcal{L}^{\prime }$
\ is injective and surjective as \ \ \ $^{t}j_{0}=$ $\rho \circ $ $^{t}j_{1}$
\ and \ $^{t}j_{0}$ ( $^{t}j_{1}$ )
are injective.
\\ 
The restriction of \ $\rho $ \ to\ $\ ^{t}j_{0}\mathcal{L}^{\prime }$
\ is bijective. Its inverse denoted\ by $\ \overline{\omega }$ \ is
bijective.
\\
(We don't know, however,  if \ $\overline{\omega }$ \ is continuous).
\\
\textbf{Example 1 :} 
\\
 $\forall $\ $u,v\in H^{1}\left( 0,1\right) ,$ \ let : \ $a\left(
u,v\right) =\int_{0}^{1}\left( u^{\prime }v^{\prime }+uv\right) $ $%
dt=\left\langle Au,v\right\rangle _{1}$ .
\\
 Then : $\forall $\ $u,v\in H^{1}\left( 0,1\right) $ , \ $a\left(
u,v\right) =\left\langle Lu,v\right\rangle _{0}$ , thus : \ $L=\rho A$ , and :
\\
 $\forall $\ $u,v\in C^{\infty }\left( 0,1\right) $ , \ $%
\int_{0}^{1}\left( u^{\prime }v^{\prime }+uv\right) $ $dt=\int_{0}^{1}\left(
-u^{^{\prime \prime }}+u\right) .v$ $dt+u^{\prime }\left( 1\right) v\left(
1\right) -u^{\prime }\left( 0\right) v\left( 0\right) $ .
\\
 i.e. \ $\left\langle Au,v\right\rangle _{1}=\left\langle \pi
Lu,v\right\rangle _{1}=\left\langle \gamma _{0}u,\gamma _{0}v\right\rangle _{%
\mathcal{T}^{\prime },\ \mathcal{T}}$ .
\\ 
If, \ $v\in \mathcal{D}\left( 0,1\right) $ \ , \ then : \ $%
\left\langle \pi Lu,v\right\rangle _{1}=\int_{0}^{1}\left( -u^{^{\prime
\prime }}+u\right) v$ $dt=\left\langle \left( -u^{^{\prime \prime
}}+u\right) ,v\right\rangle _{\mathcal{D}^{\prime }\left( 0,1\right) ,%
\mathcal{D}\left( 0,1\right) }$ .
\\ 
 \textbf{Example 2 :}
\\
 Let \ $p\in \mathbb{N}^{\ast }$ .
\\ 
 $\forall $\ $u,v\in \mathcal{B}_{1}$, $\mathcal{\ }h\left( u,v\right)
=\int_{0}^{1}\left( \left\vert u^{\prime }\right\vert ^{p-1}sign\left(
u^{\prime }\right) .v^{\prime }\right) $ $dt=\left\langle Gu,v\right\rangle
_{1}$ \ and : \ 
\\
$\qquad \forall $\ $u\in \mathcal{B}_{1}$, $\forall $\ $v\in \mathcal{B}_{0}$
, $\ \mathcal{\ }h\left( u,v\right) =\int_{0}^{1}\left( \left\vert u^{\prime
}\right\vert ^{p-1}sign\left( u^{\prime }\right) .v^{\prime }\right) $ $%
dt=\left\langle \rho \text{ }Gu,v\right\rangle _{0}$ where :
\\ 
$\rho $ $Gu$ \ is the restriction of $\ $ $Gu$ \ to \ $%
\mathcal{B}_{0}^{^{\prime }}$ .
\\ 
 Formally, we can write :%
\begin{equation*}
\left\langle Gu,v\right\rangle _{1}=\left\langle \pi \rho \text{ }%
Gu,v\right\rangle _{1}+\left\langle \gamma _{G\text{ }}u,\gamma
_{0}v\right\rangle _{\mathcal{T}^{\prime },\ \mathcal{T}}
\end{equation*}
\\
 We do not need to precise, here, the definition of \ $\mathcal{%
B}_{0}$ \ and $\mathcal{B}_{1}$ .
\\  However, if \ $p>1$ $,\ \mathcal{B}_{1}=\left\{ u\in \mathcal{D}%
^{\prime }\left( 0,1\right) \text{ ; }u,u^{\prime }\in L^{p}\left(
0,1\right) \text{ and }u\left( 0\right) =0\right\} $
\\ 
 and \ \ $\mathcal{B}_{0}=\left\{ v\in \mathcal{B}_{1}\text{ ; }\gamma
_{1}v\text{\ }=0\right\} $ \ \ with : \ $\forall $\ $v\in C^{1}\left(
0,1\right) $ \ , \ $\gamma _{1}v$\ $=v^{\prime }\left( 1\right) $
\\ 
 we deduce that, if $\ u$ \ is sufficiently regular, $\gamma _{G\text{ 
}}u=\left( \left\vert u^{\prime }\right\vert ^{p-1}sign\left( u^{\prime
}\right) \right) \left( 1\right) $ .
\\
\qquad (cf. " Quelques m\'{e}thodes de r\'{e}solution de probl\`{e}mes aux
limites non lin\'{e}aires",
\qquad de J.L. Lions, p.25). 
\section{ Green's formula.}
 We recall that $\overline{\omega }$ \ is a linear mapping from \ \ $\
^{t}j_{0}\mathcal{L}^{\prime }$ \ into \ $\mathcal{B}_{1}^{\prime }$ \ such
that :
  $$\forall \, l^{\prime }\in \mathcal{L}^{\prime },\quad \overline{\omega }\, ^{t}j_{0}l^{\prime }=\, ^{t}j_{1}l^{\prime }.$$
 Let \ $G$ \ be an operator  (which is linear or not)  continuous from
\ $\mathcal{B}_{1}$ \ into \ $\mathcal{B}_{1}^{\prime }$ .
  We set : 
\begin{equation*}
\mathcal{B}_{G}=\left\{ u\in \mathcal{B}_{1}\text{ ; }\rho \text{ }G\text{ }%
u\in \ ^{t}j_{0}\mathcal{L}^{\prime }\right\}
\end{equation*}
 Then : 
\begin{eqnarray*}
\forall u &\in &\mathcal{B}_{G}\text{ , \ }\exists \text{ }l^{\prime }\in 
\mathcal{L}^{\prime }\text{ \ such that : }^{t}j\left( G\text{ }u-\
^{t}j_{1}l^{\prime }\right) =\theta _{\mathcal{B}_{0}^{\prime }}\text{\ } \\
&&\text{where \ }\theta _{\mathcal{B}_{0}^{\prime }}\text{ \ is the neutral
element of \ }\mathcal{B}_{0}^{\prime }
\end{eqnarray*}
 Thus, we deduce that :%
\begin{equation*}
G\text{ }u-\ ^{t}j_{0}\text{ }l^{\prime }=G\text{ }u-\overline{\omega }\
^{t}j_{0}\text{ }l^{\prime }\in \ker \text{ }^{t}j=\func{Im}^{t}\gamma _{0}
\end{equation*}
 So : \ 
\begin{equation*}
G\text{ }u-\overline{\omega }\ ^{t}j_{0}\text{ }l^{\prime }=G\text{ }u-%
\overline{\omega }\ \rho \text{ }G\text{ }u=\text{ }^{t}\gamma _{0}z^{\prime
}\text{ , with \ }z^{\prime }\in \mathcal{T}^{\prime }
\end{equation*}
  We shall set, in the following :%
\begin{equation*}
\forall u\in \mathcal{B}_{G}\text{ , }\gamma _{G\text{ }}u=\text{ }%
^{t}\gamma _{0}^{-1}\left( I_{1}^{\prime }-\overline{\omega }\ \rho \right) G%
\text{ }u\in \mathcal{T}^{\prime }
\end{equation*}
  where \ $I_{1}^{\prime }$ \ is the identity of $\mathcal{B}%
_{1}^{\prime }$ .
 \\
  \textbf{Remarks :}
\begin{itemize}
\item[(*)] \ $\gamma _{G\text{ }}$ is a mapping from \ 
$\mathcal{B}_{G}$ \ into \ $\mathcal{T}^{\prime }$ \ \ which is not necessarily linear,
  but \ $^{t}\gamma _{0}^{-1}\left( I_{1}^{\prime }-\overline{%
\omega }\ \rho \right) $ \ is a linear mapping which is definite on \ $%
G\left( \mathcal{B}_{G}\right) $ .
\\
\item[(**)] \ $\left( u\in \mathcal{B}_{G}\right) \iff \chi
_{1}^{\_1}J_{1}^{\ast }G$ $u\in \mathcal{B}_{J_{1}}$ .
\\ Then :%
\begin{equation*}
\forall u\in \mathcal{B}_{G}\text{ , }\forall y\in \mathcal{B}_{1}\text{ , \ 
}\left\langle G\text{ }u,y\right\rangle _{1}=\left\langle \overline{\omega }%
\ \rho \text{ }G\text{ }u,y\right\rangle _{1}+\left\langle \gamma _{G\text{ }%
}u,\gamma _{0\text{ }}y\right\rangle _{\mathcal{T}^{\prime },\mathcal{T}}
\end{equation*} 
 Thus, we have the following scheme :
$$
 \begin{array}{ccccccccccc}
&&&& \mathcal{T}^{\prime }&\longrightarrow &\mathcal{T}&&&&
\\ 
&&&&\downarrow&&\uparrow &&&&
\\ 
E^{\prime}&
\longrightarrow&\mathcal{L}^{\prime} & \longrightarrow &  
\mathcal{B}_{1}^{\prime} & \longrightarrow & 
\mathcal{B}_{1} & \longrightarrow & \mathcal{L}
&\longrightarrow E
\\  
&&&\searrow & \downarrow & &\uparrow & \nearrow& &&
\\ 
&&&&\mathcal{B}_{0}^{\prime }&\longrightarrow & \mathcal{B}_{0}&&&&
\end{array}  
$$ 
 So : 
\begin{eqnarray*}
l_{0} &=&l_{1}\circ j\text{ \ , \ }^{t}l_{0}=\rho \circ \text{ }^{t}l_{1}%
\text{ , \ }l_{1}=k\circ j_{1}\text{ , \ }l_{0}=k\circ j_{0} 
\\
\text{with \ }\forall e^{\prime } &\in &E^{\prime }\text{ , \ }\overline{%
\omega }\text{ }^{t}l_{0}e^{\prime }=\text{ }^{t}l_{1}e^{\prime }\text{\ }
\end{eqnarray*}
\end{itemize}
\textbf{Example : }%
\begin{equation*}
\mathcal{B}_{1}=W^{1,p}\left( \Omega \right) \text{ \ , \ }\mathcal{B}%
_{0}=W_{0}^{1,p}\left( \Omega \right) \text{ \ , }\ \ \mathcal{L=}\text{ }%
L^{p}\left( \Omega \right) \text{ \ , \ }E=\mathcal{D}^{\prime }\left(
\Omega \right)
\end{equation*}
 \\
  Now, we suppose that \ $j_{1\text{ }}$(resp. $j_{0},k$) is a
weakly dense injection 
 (and weakly continuous) from $\ \mathcal{B}_{1}$ \ into \ $\mathcal{L}
$ \ (resp. $\mathcal{B}_{0}$ \ into \ $\mathcal{L}$ \ , $\mathcal{L}$ \ into 
$E$ ).
\\ 
 Then : \ $^{t}j_{1}$ (resp. $^{t}j_{0}$ , $^{t}k$ ) is injective and
its image is weakly (strongly) dense
in $\ \ \mathcal{B}_{1}^{\prime }$ (resp. $\mathcal{B}_{0}^{\prime }$
, $\mathcal{L}^{\prime }$) when those vector spaces are Banach spaces
(reflexive or not).
\\  Let us set :%
\begin{equation*}
\mathcal{B}_{G,E}=\left\{ u\in \mathcal{B}_{1}\text{ ; }\rho \text{ }G\text{ 
}u\in \ ^{t}l_{0}\text{ }E^{\prime }\right\} \subset \mathcal{B}_{G}
\end{equation*}
  So :%
\begin{eqnarray*}
\forall u &\in &\mathcal{B}_{G,E}\text{ , \ }\exists \text{ }e^{\prime }\in
E^{\prime }\text{ , \ such that : \ }^{t}j\left( G\text{ }u-\
^{t}l_{1}l^{\prime }\right) =\theta _{\mathcal{B}_{0}^{\prime }}\text{\ } \\
&\implies &G\text{ }u-\ ^{t}l_{1}e^{\prime }=G\text{ }u-\overline{\omega }%
_{E}\ ^{t}l_{0}\text{ }e^{\prime }\in \ker \text{ }^{t}j=\func{Im}^{t}\gamma
_{0} \\
&\implies &G\text{ }u-\overline{\omega }\ ^{t}l_{0}\text{ }e^{\prime }=G%
\text{ }u-\overline{\omega }_{E}\rho \text{ }G\text{ }u=^{t}\gamma _{0}\text{
}z^{\prime }\text{ , }z^{\prime }\in \mathcal{T}^{\prime }
\end{eqnarray*}
 Now, we set :%
\begin{equation*}
\forall u\in \mathcal{B}_{G,E}\text{ , \ }\gamma _{G\text{ }}u=\text{ }%
^{t}\gamma _{0}^{-1}\left( I_{1}^{\prime }-\overline{\omega }_{E}\ \rho
\right) G\text{ }u\in \mathcal{T}^{\prime }
\end{equation*}
 \\
  \textbf{Remark :}
\\ 
 $\forall $ $e^{\prime }\in E^{\prime }$ , \ $\overline{%
\omega }_{E}$ $^{t}j_{0}$ $^{t}k$ $e^{\prime }=$ $^{t}j_{1}$ $^{t}k$ $%
e^{\prime }=$\ $\overline{\omega }$ $^{t}j_{0}$ $^{t}k$ $e^{\prime }$
\\ 
Thus : 
 $\overline{\omega }_{E}$ \ is the restriction of \ $%
\overline{\omega }$ \ to \ \ $^{t}j_{0}$ $^{t}k$ $E^{\prime }$ \ which is
contained in \ $^{t}j_{0}\mathcal{L}^{\prime }$, and : 
\begin{equation*}
\forall u\in \mathcal{B}_{G,E}\text{ , \ }\forall y\in \mathcal{B}_{1}\text{
, \ }\left\langle Gu,y\right\rangle _{1}=\left\langle \overline{\omega }%
_{E}\rho \text{ }G\text{ }u,y\right\rangle _{1}+\left\langle \gamma _{G,E%
\text{ }}u,\gamma _{0\text{ }}y\right\rangle _{\mathcal{T}^{\prime },%
\mathcal{T}}
\end{equation*}
 where \ $\gamma _{G,E\text{ }}$ \ is the restriction of \ \ $%
\gamma _{G\text{ \ }}$ to \ \ $\mathcal{B}_{G,E}\subset \mathcal{B}_{G}$ .
\section{ An order on \ $\mathcal{L}$ \ and its induced order on \ $%
\mathcal{B}_{1}$ (resp. $\mathcal{B}_{0}$ ).}
 Let \ \ $C^{+\text{ \ }}$ be a closed convex cone in \ $%
\mathcal{L}$ \ \ with \ $\theta _{\mathcal{L}}$ \ as\ its vertex.
\\ 
 Then : 
\begin{equation*}
\mathcal{L=}\text{ }C^{+}-C^{+}\text{ \ and \ }\theta _{\mathcal{L}%
}=C^{+}\cap \left( -C^{+}\right)
\end{equation*}
 Let :%
\begin{equation*}
C_{m}^{+}=\left\{ u\in \mathcal{B}_{m}\text{ ; \ }j_{m}u\in C^{+}\right\} 
\text{ \ for \ }m=0,1
\end{equation*}
 We have : 
\begin{equation*}
\mathcal{B}_{m}=C_{m}^{+}-C_{m}^{+}\text{ \ and \ }\theta _{\mathcal{B}%
_{m}}=C_{m}^{+}\cap \left( -C_{m}^{+}\right) \text{ \ for \ }m=0,1
\end{equation*}
 \\
  \textbf{Hypotheses :}
\\ 
 We suppose that :
\begin{itemize}
\item[$\left( \ast \right) $] for \ $m=0,1$ \ , \ \ $j_{m}C_{m}^{+}$
\ is dense in \ $C^{+}$ (for the topology of \ $\mathcal{L}$) .
\\
\item[ $\left( \ast \ast \right) $] $\mathcal{B}_{m}$ \ is nested \
for the order induced by that of \ $\mathcal{L}$ . \ 
\end{itemize}
If \ $u\in \left[ C^{+}\right] $ (resp. $C_{1}^{+},$ $%
C_{0}^{+} $ ), we shall write : $u\geq 0$
  (and $u\leq 0$ \ if \ $u\in \left( -C^{+}\right) =C^{-}$
(resp. $\left( -C_{m}^{+}=C_{m}^{-}\right) $ for \ $m=0,1$ )
\\ 
 We shall denote by \ $D^{+}$ (resp. $D_{1}^{+}$ , $D_{0}^{+}$
) the polar cone of \ $C^{+}$ (resp. $C_{1}^{+}$ , $C_{0}^{+}$ )
  for the duality between \ $\mathcal{L}$ and \ $\mathcal{L}%
^{\prime }$ $\ $(resp.~$\mathcal{B}_{1}$ and \ $\mathcal{B}_{1}^{\prime }$ , 
$\mathcal{B}_{0}$ and \ $\mathcal{B}_{0}^{\prime }$ ).
 \\
  So :%
\begin{equation*}
D^{+}=\left\{ x^{\ast }\in \mathcal{L}^{\prime }\text{ };\text{ }\forall
x\in \ C^{+}\text{ , }\left\langle x^{\ast },x\right\rangle \geq 0\right\}
\end{equation*}
  $D^{+}$ (resp. \ $D_{m}^{+}$ \ for \ $m=0,1$ ) is a weakly
(and strongly) closed cone as $\mathcal{B}$ \ 
 (resp. $\mathcal{B}_{m}$ \ \ for \ $m=0,1$) \ is reflexive.
\\ 
If \ $u^{\ast }\in D^{+}$ (resp. \ $D_{m}^{+}$ \ for \ $%
m=0,1$ ) , we shall write \ $u^{\ast }\geq 0$ \ and
  $u^{\ast }\leq 0$ \ if \ $u\in \left( -D^{+}\right)
=D^{-}$ (resp. $\left( -D_{m}^{+}=D_{m}^{-}\right) $ for \ $m=0,1$) .
 \begin{proposition}
  $D_{m}^{+}$ \ is the adherence of \ \ $^{t}j_{m}D^{+}$
\ in \ $\mathcal{B}_{m}^{\prime }$ for \ $m=0,1$ .
\end{proposition}
\begin{proof}
  First, we remark that : \ $^{t}j_{m}D^{+}\subset $\ $%
D_{m}^{+}$ \ (we have : $^{t}j_{0}D^{+}=\rho $ $^{t}j_{1}D^{+}$ ).
\\ 
Indeed : \ $\forall l^{\prime }\in D^{+}$ , \ $\forall u\in
C_{m}^{+}$ , \ $\left\langle ^{t}j_{m}l^{\prime },u\right\rangle
_{m}=\left\langle l^{\prime },\text{ }j_{m}u\right\rangle _{\mathcal{L}%
^{\prime },\mathcal{L}}\geq 0$ .
\\ 
 Let us suppose that : \ \ $^{t}j_{m}D^{+}\subsetneq $\ $%
D_{m}^{+}$ \ \ and \ let : \ $u^{\prime }\in D_{m}^{+}\backslash \overline{\
^{t}j_{m}D^{+}}$ .
 \\
  As \ \ $\overline{\ ^{t}j_{m}D^{+}}$ $\ $\ is a closed
convex cone (with \ \ $\theta _{\mathcal{B}_{m}^{^{\prime }}}$ \ as its
vertex)
\\ 
 and as \ \ $\mathcal{B}_{m}$ \ is reflexive, there
exists \ $u\in \mathcal{B}_{m}$ \ such that :%
\begin{equation*}
\ \forall l^{\prime }\in D^{+}\text{ \ , \ }\left\langle ^{t}j_{m}l^{\prime
},u\right\rangle _{m}\geq 0\text{ \ and : \ }\left\langle u^{\prime
},u\right\rangle _{m}<0
\end{equation*}
  Thus, we have :%
\begin{equation*}
\ \forall l^{\prime }\in D^{+}\text{ \ , \ }\left\langle l^{\prime },\text{ }%
j_{m}u\right\rangle \geq 0\text{ \ and : \ }\left\langle u^{\prime
},u\right\rangle _{m}<0
\end{equation*}
 So : 
\begin{equation*}
\left( \text{ }j_{m}u\geq 0\text{ \ , \ }u\geq 0\text{ \ and : \ }%
\left\langle u^{\prime },u\right\rangle _{m}<0\right) \text{ \ \ what is
impossible.}
\end{equation*}
\end{proof}
 
 \begin{lemma}
 $\overline{\omega }$ \ is an homomorphism for the oder
structures which are definite on \ $^{t}j_{0}\mathcal{L}^{\prime }$ and $^{t}j_{1}\mathcal{L}^{\prime }$ \ respectively by \ $%
^{t}j_{0}D^{+}$ \ and \ $^{t}j_{1}D^{+}$ .
\\ 
 Moreover : \ $\overline{\omega }\left(
^{t}j_{0}D^{+}\right) =\ ^{t}j_{1}D^{+}$ .
\end{lemma}  
\begin{proof}
 $\left( l^{\prime }\in \mathcal{L}^{\prime }\text{ , }%
^{t}j_{m}l^{\prime }\geq 0\text{\ }\right) \iff \left( l^{\prime }\in 
\mathcal{L}^{\prime }\text{ , \ }l^{\prime }\geq 0\right) $ \ \ for \ $m=0,1$
\ because
\\ 
 $\left( \text{ }^{t}j_{m}l^{\prime }\geq 0\right) \iff \left(
\forall v\in C_{m}^{+}\text{ , \ }\left\langle ^{t}j_{m}l^{\prime
},v\right\rangle _{m}=\left\langle l^{\prime },\text{ }j_{m}v\right\rangle
_{m}\geq 0\right) $
\\ 
 $\iff \left( \forall u\in C^{+}\text{ , \ }\left\langle
l^{\prime },u\right\rangle \geq 0\right) $ \ (because \ $^{t}j_{m}C^{+}$ \
is dense in \ $C^{+}$).
\\  Thus : $\left( ^{t}j_{0}l^{\prime }\geq 0\right) \implies
\left( \ \overline{\omega }\ ^{t}j_{0}l^{\prime }=\text{ }^{t}j_{1}l^{\prime
}\geq 0\right) $ .
\end{proof}
\begin{lemma} 
$\forall u\in C_{1}^{+}$ , \ $\forall v^{\prime }\in
\rho $ $D_{1}^{+}$\ , \ $\forall w^{\prime }$\ $\in \left( \rho \right)
^{-1}v^{\prime }$ ,%
\begin{equation*}
Sup\left\{ \left\langle v^{\prime },\text{ }v\right\rangle _{0}\text{ ; \ }%
0\leq v\leq u\text{ , \ }v\in \mathcal{B}_{0}\right\} \leq \left\langle
w^{\prime },\text{ }u\right\rangle _{1}\text{ .}
\end{equation*}
\end{lemma} 
\begin{proof}%
\begin{equation*}
\text{\ }\forall v\in \mathcal{B}_{0}\text{ , }0\leq v\leq u\text{ , }%
\left\langle v^{\prime },\text{ }v\right\rangle _{0}=\left\langle \rho
w^{\prime },\text{ }v\right\rangle _{0}=\left\langle w^{\prime },\text{ }%
jv\right\rangle _{1}\text{\ }\leq \left\langle w^{\prime },\text{ }%
u\right\rangle _{1}
\end{equation*}
 \end{proof}
 \begin{lemma}
   $\forall u\in C_{1}^{+}$ , \ $\forall v^{\prime }\in $ $%
^{t}j_{0}D^{+}=\rho $ $^{t}j_{1}D_{1}^{+}$\ ,
  $\left\langle \ \overline{\omega }v^{\prime
},u\right\rangle _{1}=Sup\left\{ \left\langle v^{\prime },\text{ }%
v\right\rangle _{0}\text{ ; \ }0\leq v\leq u\text{ }\right\} $
\end{lemma}
\begin{proof}  
Let \ $v\in C_{0}^{+}$ \ such that : \ \ $u-jv\in
C_{1}^{+}$ .Then :%
\begin{equation*}
\left\langle \ \overline{\omega }v^{\prime },u-jv\right\rangle
_{1}=\left\langle \text{ }^{t}j_{1}l^{\prime },u-jv\right\rangle
_{1}=\left\langle l^{\prime },\text{ }j_{1}\left( u-jv\right) \right\rangle
_{\mathcal{L}^{\prime },\mathcal{L}}\geq 0\text{ \ (where  }l^{\prime }\in
D^{+}\text{ )}
\end{equation*}
  Thus : 
  $$\left\langle \ \overline{\omega }%
v^{\prime },u\right\rangle _{1}\geq \left\langle v^{\prime },v\right\rangle
_{0}.$$
 On the other hand, we know that there exists a sequence  $%
\left( v_{m}\text{ ; }m\in \mathbb{N}\right) $ 
 contained in \ $C_{0}^{+}$ \ such that :%
\begin{equation*}
\forall p\in \mathbb{N}\text{ , }u-jv_{p}\in C_{1}^{+}\text{ \ \ and \ \ }%
\lim_{p\rightarrow \infty }\text{ }\left\Vert j_{0}v_{p}-j_{1}u\right\Vert _{%
\mathcal{L}}=0
\end{equation*}
 We deduce that :%
\begin{equation*}
\forall p\in \mathbb{N}\text{ , }\left\langle \ \overline{\omega }v^{\prime
},u\right\rangle _{1}\geq \left\langle v^{\prime },v_{p}\right\rangle _{0}
\end{equation*}
  and :%
\begin{eqnarray*}
\lim_{p\rightarrow \infty }\text{ }\left\langle v^{\prime
},v_{p}\right\rangle _{0} &=&\lim_{p\rightarrow \infty }\text{ }\left\langle 
\text{ }^{t}j_{0}l^{\prime },v_{p}\right\rangle _{0}=\lim_{p\rightarrow
\infty }\text{ }\left\langle \text{ }l^{\prime },\text{ }j_{0}v_{p}\right%
\rangle _{_{\mathcal{L}^{\prime },\mathcal{L}}} \\
&=&\text{ }\left\langle \text{ }l^{\prime },\text{ }j_{1}u\right\rangle _{_{%
\mathcal{L}^{\prime },\mathcal{L}}}=\left\langle \text{ }^{t}j_{1}l^{\prime
},u\right\rangle _{1}=\left\langle \ \overline{\omega }v^{\prime
},u\right\rangle _{1}
\end{eqnarray*}
\end{proof}
\begin{proposition}   $\forall u\in C_{1}^{+}$ , \ $\forall v^{\prime }\in $ $%
^{t}j_{0}D^{+}=\rho $ $^{t}j_{1}D_{1}^{+}$\ ,
 $$\left\langle \ \overline{\omega }v^{\prime
},u\right\rangle _{1}=Sup\left\{ \left\langle v^{\prime },v\right\rangle _{0}%
\text{ ; }0\leq v\leq u\text{\ }\right\} =Inf\left\{ \left\langle w^{\prime
},u\right\rangle _{1}\text{ ; \ }u^{\prime }\in \left( \rho \right)
^{-1}v^{\prime }\right\} $$
\end{proposition}
\begin{proof}  deduced from lemmas 2 and 3 and from the fact that : \ $%
\ \overline{\omega }v^{\prime }\in \left( \rho \right) ^{-1}v^{\prime }$ .
\end{proof} 
 \textbf{References. }
 \vspace{0.3cm}
 
 [1] B. Hanouzet, J.L. Joly.
 \textit{Applications bilinéaires sur certains sous espaces de Sobolev.} CRAS Paris, t. 294 (1982).
\vspace{0.2cm} 

[2]  J.L. Lions.  Quelques m\'{e}thodes de r\'{e}solution de probl\`{e}mes aux
limites non lin\'{e}aires. Dunod, Gauthiers Villard, 1969.
 
 \end{document}